\documentclass[10pt]{article}

\title{Lifting Cubic Realizations of Weak Orders in Types A and B}
\author{Daria Poliakova}
\date{}

\usepackage[a4paper,margin=2.8cm]{geometry}
\usepackage{amsmath,amssymb,amsthm}
\usepackage{mathtools}
\usepackage{enumitem}
\usepackage[hidelinks]{hyperref}

\usepackage{tikz}
\usetikzlibrary{arrows.meta,calc,positioning}

\setlist[itemize]{leftmargin=2em}
\setlist[enumerate]{leftmargin=2.2em}

\usepackage{accents}
\newlength{\dhatheight}

\newtheorem{theorem}{Theorem}[section]
\newtheorem{proposition}[theorem]{Proposition}
\newtheorem{lemma}[theorem]{Lemma}
\newtheorem{corollary}[theorem]{Corollary}

\theoremstyle{definition}
\newtheorem{definition}[theorem]{Definition}
\newtheorem{remark}[theorem]{Remark}

\newcommand{\Inv}{\mathrm{Inv}}

\begin{document}

\maketitle

\begin{abstract}
We study cubic realizations of posets compatible with projection maps, meaning that the projection is represented by deletion of the last coordinate. For cylindrical projections, we introduce the pre-Reeb graph and the augmented pre-Reeb graph, which control compatible cubic lifts and compatible order-embedding cubic lifts, respectively. We apply this construction to the deletion towers in weak order of types \(A\) and \(B\). The pre-Reeb graphs are the \(1\)-skeleta of, respectively, cubes and certain zonotopes. In both cases, the augmented pre-Reeb graphs have reachability posets that are total orders, yielding combinatorial uniqueness of the compatible order-embedding cubic coordinates.
\end{abstract}

\section*{Introduction}

Many families of combinatorial polytopes admit embeddings of their vertex sets into integer cubes. Two standard examples are bracket vectors for associahedra and Lehmer codes for permutahedra. In the associahedral case, bracket vectors give an order embedding of the Tamari lattice \cite{HuangTamari1972}. In the permutahedral case, Lehmer codes give a cubical embedding of the weak order, but not an order embedding of it. However, the order dimension of the weak order is $n-1$ as proved in \cite{Flath1993MultinomialLattices}, and indeed the inversions counted to obtain Lehmer codes can be weighted exponentially, to obtain an order embedding -- this version has appeared, for example, in \cite{Reading2003}, and in yet unpublished work of Jonah Berggren on cubic coordinates for framing lattices in general. People have considered related cubic coordinate systems for other combinatorial objects, including Tamari intervals, permutrees, and Hochschild polytopes \cite{Combe2023,Jimenez2023,PilaudPoliakova2025}.

Many such families also come with natural projection maps between consecutive ranks. In this paper, cubic realizations of $P$ and $Q$ are called compatible with a projection $\pi: P \to Q$ if the projection is represented on coordinates by deletion of the last coordinate. This leads to the following questions. Given a projection \(\pi:P\to Q\) and a cubic realization of \(Q\), when does there exist a compatible cubic realization of \(P\)? If the realization of \(Q\) is an order embedding, when does there exist a compatible lift which is again an order embedding? We study these questions for cylindrical projections, i.e. projections whose fibers are total orders, such that the bottom and top elements of the fibers form two horizontal copies of the Hasse diagram of the base.

For a cylindrical projection \(\pi:P\to Q\), we introduce two directed graphs: the pre-Reeb graph \(R(\pi)\) and the augmented pre-Reeb graph \(\widehat R(\pi)\). The pre-Reeb graph controls compatible cubic lifts. The augmented pre-Reeb graph controls compatible order-embedding cubic lifts. More precisely, if \(R(\pi)\) is acyclic, then any cubic realization of \(Q\) extends to a cubic realization of \(P\) by one new coordinate. If \(\widehat R(\pi)\) is acyclic, then any order-embedding cubic realization of \(Q\) extends to an order-embedding cubic realization of \(P\) by one new coordinate. 

Recall that the Dushnik--Miller dimension of a poset $P$ is minimal $d$ such that there exists an order embedding $P \hookrightarrow \mathbb{R}^d$ \cite{DushnikMiller1941}. For a cylindrical tower over a point, 
\[
\mathrm{pt} = P_0 \xleftarrow{\pi_1} P_1 \xleftarrow{\pi_2} P_2 \xleftarrow{} \cdots ,
\]
such that all augmented pre-Reeb graphs are acyclic, we get $\dim P_n = n$.

We apply our constructions to the deletion towers in weak order of types \(A\) and \(B\),
\[
\mathrm{pt} = S_1 \xleftarrow{} S_2 \xleftarrow{} \cdots,
\qquad
\mathrm{pt} \xleftarrow{} W_1 \xleftarrow{} W_2 \xleftarrow{} \cdots.
\]
Pre-Reeb graphs in type \(A\) are skeleta of cubes, or Hasse diagrams of Boolean lattices. Pre-Reeb graphs in type \(B\) are skeleta of certain zonotopes; this description appears to be new and unexpected. For augmented pre-Reeb graphs, both types yield total orders as reachability posets. It follows that the compatible order-embedding coordinates in these towers are combinatorially unique.

The resulting cubic coordinates for both types appear already in \cite{Reading2003}, where Reading proves that the weak orders of types \(A\) and \(B\) have order dimension equal to the rank. The novel input of the current paper is the lifting formalism via pre-Reeb and augmented pre-Reeb graphs, the explicit computation of these structures for the deletion towers in types A and B, and the resulting uniqueness statements for projection-compatible coordinates. An intended direction for further work is to apply the current formalism to permutoassociahedra \cite{Kapranov1993} and to higher categorical associahedra \cite{Bottman2019, BackmanBottmanPoliakova2024}.

\subsection*{Acknowledgements}

The ideas leading to this paper arose during the CIRM event \emph{Beyond Permutahedra and Associahedra}. I benefited from discussions with Sergey Arkhipov, Spencer Backman, Jonah Berggren, C{\'e}sar Ceballos, Hung Hoang, Vincent Pilaud -- and ChatGPT. I am supported by the Collaborative Research Center CRC 1624 \emph{Higher Structures, Moduli Spaces and Integrability} at the University of Hamburg.

\section{Setup}

\subsection{Cubic realizations and (augmented) pre-Reeb graphs}

augmented
\begin{definition}
A \emph{cubic realization} of a poset \(P\) is an embedding \(c:P\to \mathbb R^d\) of its Hasse diagram such that for every cover \(x\to y\), the points \(c(x)\) and \(c(y)\) differ in exactly one coordinate, and in that coordinate one has \(c(x)<c(y)\).
An \emph{order-embedding cubic realization} is a cubic realization such that
\[
x\le y \quad\Longleftrightarrow\quad c(x)\le c(y)
\]
coordinatewise.
\end{definition}

\begin{definition}
Let \(P,Q\) be posets and \(\pi:P\to Q\) a map. We say that \(\pi\) is a \emph{cylindrical projection} if:
\begin{enumerate}[label=\arabic*)]
    \item for every cover \(v\to w\) in \(P\), either \(\pi(v)=\pi(w)\) or \(\pi(v)\to\pi(w)\) is a cover in \(Q\);
    \item for every \(q\in Q\), the fiber \(\pi^{-1}(q)\subset P\) is a total order with at least 2 elements;
\item if \(b(q)\) and \(t(q)\) denote respectively the minimal and maximal elements of \(\pi^{-1}(q)\), then the maps
    \[
    b,t:Q\to P
    \]
    identify the Hasse diagram of \(Q\) with induced subgraphs of the Hasse diagram of \(P\). Equivalently, for every cover \(q\to q'\) in \(Q\), the relations
    \[
    b(q)\to b(q'),\qquad t(q)\to t(q')
    \]
    are covers in \(P\).
\end{enumerate}
\end{definition}

In most of the constructions and propositions of this paper, just the conditions \(1)\) and \(2)\) in the definition of cylindrical projection are used. Condition \(3)\) is only used to obtain the equality in the dimension statement for towers in Proposition \ref{lem:dimension-cylindrical}. It is however satisfied in all the examples of interest; we thus keep it for aesthetic reasons.

\begin{definition}
Let \(\pi:P\to Q\) be a projection, let \(c_Q:Q\to\mathbb R^d\) and \(c_P:P\to\mathbb R^{d+1}\) be cubic realizations. We say that \(c_P\) is \emph{compatible} with \(c_Q\) and \(\pi\) if
\[
c_P(x)=(c_Q(\pi(x)),h(x))
\]
for some function \(h:P\to\mathbb R\).
\end{definition}

\begin{definition}
Let \(\pi:P\to Q\) be a cylindrical projection. A cover \(v\to w\) in \(P\) is
\begin{itemize}
    \item \emph{horizontal} if \(\pi(v)\to\pi(w)\) is a cover in \(Q\);
    \item \emph{vertical} if \(\pi(v)=\pi(w)\).
\end{itemize}
Let \(\sim_h\) be the equivalence relation on \(P\) generated by undirected adjacency along horizontal covers.
\end{definition}

\begin{definition}
Let \(\pi:P\to Q\) be a cylindrical projection.
The \emph{pre-Reeb graph} \(R(\pi)\) is the directed graph whose vertices are the \(\sim_h\)-equivalence classes \([v]\), with an edge
\[
[v]\to[w]
\]
whenever there exist representatives \(v'\in[v]\), \(w'\in[w]\) such that \(v'\to w'\) is a vertical cover in \(P\).
\end{definition}

\begin{definition}\label{def:auxiliary-edges}
Let \(\pi:P\to Q\) be a cylindrical projection.
The \emph{augmented pre-Reeb graph} \(\widehat{R}(\pi)\) is obtained from \(R(\pi)\) by adding an auxiliary edge \([v]\to[w]\) whenever there exist representatives \(v'\in[v]\), \(w'\in[w]\) such that
\[
\pi(w')<\pi(v')\ \text{ in }Q,
\qquad
v' \text{ and } w' \text{ are incomparable in }P.
\]
\end{definition}

\begin{definition}
Let \(G\) be an acyclic directed graph. Its \emph{reachability poset} is the poset on \(V(G)\) given by
\[
x\le_G y \quad\Longleftrightarrow\quad \text{there exists a directed path }x\rightsquigarrow y\text{ in }G.
\]
\end{definition}

\begin{definition}
Assume \(R(\pi)\) is acyclic. The \emph{Reeb poset} of \(\pi\) is the reachability poset of \(R(\pi)\). We denote it by \(\mathsf{Reeb}(\pi)\).
\end{definition}

\begin{proposition}\label{prop:cubic-extension}
Let \(c_Q:Q\to\mathbb R^d\) be a cubic realization. Then compatible cubic realizations
\[
c_P:P\to\mathbb R^{d+1}
\]
are in bijective correspondence with functions
\[
h:V(R(\pi))\to\mathbb R
\]
strictly increasing along edges.
In particular, such a compatible cubic realization exists if and only if \(R(\pi)\) is acyclic.
\end{proposition}

\begin{proof}
A compatible cubic realization has the form
\[
c_P(x)=(c_Q(\pi(x)),h_P(x))
\]
for some \(h_P:P\to\mathbb R\). Along a horizontal cover, the first \(d\) coordinates already change, so the last one does not; hence \(h_P\) is constant on horizontal classes. Thus it descends to a function
\[
h:V(R(\pi))\to\mathbb R.
\]
Along a vertical cover, the first \(d\) coordinates are equal, so the last one must strictly increase. Hence \(h\) is strictly increasing along edges.

Conversely, let \(h:V(R(\pi))\to\mathbb R\) be strictly increasing along edges, and define
\[
c_P(x):=(c_Q(\pi(x)),h([x])).
\]
Horizontal covers change only the first \(d\) coordinates, and vertical covers only the last one, so \(c_P\) is a cubic realization. It is injective: if \(c_P(x)=c_P(y)\), then \(\pi(x)=\pi(y)\) and \(h([x])=h([y])\). If \(x\neq y\), the fiber is a total order, so one gets a nontrivial vertical chain from one to the other, hence a directed path in \(R(\pi)\) from \([x]\) to \([y]\), contradicting strict increase of \(h\). Thus \(x=y\).

The two constructions are inverse to each other. Finally, such an \(h\) exists if and only if \(R(\pi)\) is acyclic.
\end{proof}

Thus, if \(h\) is any increasing function on \(\mathsf{Reeb}(\pi)\), then the function \(x\mapsto h([x])\) on \(P\) has \(R(\pi)\) as its Reeb graph; this explains the terminology.

\begin{definition}
Assume \(\widehat{R}(\pi)\) is acyclic. The \emph{augmented Reeb poset} of \(\pi\) is the reachability poset of \(\widehat{R}(\pi)\). We denote it by \(\widehat{\mathsf{Reeb}}(\pi)\).
\end{definition}

\begin{proposition}\label{prop:good-cubic-extension}
Let \(c_Q:Q\to\mathbb R^d\) be an order-embedding cubic realization. Then compatible order-embedding cubic realizations
\[
c_P:P\to\mathbb R^{d+1}
\]
are in bijective correspondence with functions
\[
h:V(\widehat{R}(\pi))\to\mathbb R
\]
strictly increasing along edges.
In particular, such a compatible order-embedding cubic realization exists if and only if \(\widehat{R}(\pi)\) is acyclic.
\end{proposition}

\begin{proof}
If \(c_P\) is a compatible order-embedding cubic realization, Proposition~\ref{prop:cubic-extension} gives a function
\[
h:V(R(\pi))\to\mathbb R
\]
strictly increasing along edges of \(R(\pi)\). Let \([x]\to[y]\) be an auxiliary edge of \(\widehat{R}(\pi)\), with witnesses \(x'\in[x]\), \(y'\in[y]\) such that \(\pi(y')<\pi(x')\) and \(x',y'\) are incomparable. Since \(c_Q\) is order-embedding, \(c_Q(\pi(y'))\le c_Q(\pi(x'))\). If \(h([y])\le h([x])\), then \(c_P(y')\le c_P(x')\), hence \(y'\le x'\), contradiction. Thus \(h([x])<h([y])\). So \(h\) is strictly increasing along edges of \(\widehat{R}(\pi)\).

Conversely, let
\[
h:V(\widehat{R}(\pi))\to\mathbb R
\]
be strictly increasing along edges. Since \(R(\pi)\) is a subgraph of \(\widehat{R}(\pi)\), Proposition~\ref{prop:cubic-extension} yields a compatible cubic realization
\[
c_P(x)=(c_Q(\pi(x)),h([x])).
\]
We show that it is order-embedding. The implication \(x\le y\Rightarrow c_P(x)\le c_P(y)\) is immediate from covers. Conversely, assume \(c_P(x)\le c_P(y)\). Then \(c_Q(\pi(x))\le c_Q(\pi(y))\), hence \(\pi(x)\le \pi(y)\). If \(x\) and \(y\) were incomparable, then \(\pi(x)\neq\pi(y)\), so in fact \(\pi(x)<\pi(y)\), and by definition of \(\widehat{R}(\pi)\) there would be an auxiliary edge \([y]\to[x]\). Thus \(h([y])<h([x])\), contradicting \(c_P(x)\le c_P(y)\). Hence \(x\le y\).

Thus the compatible order-embedding cubic realizations are exactly those coming from functions \(h\) strictly increasing along edges of \(\widehat{R}(\pi)\). Finally, such an \(h\) exists if and only if \(\widehat{R}(\pi)\) is acyclic.
\end{proof}

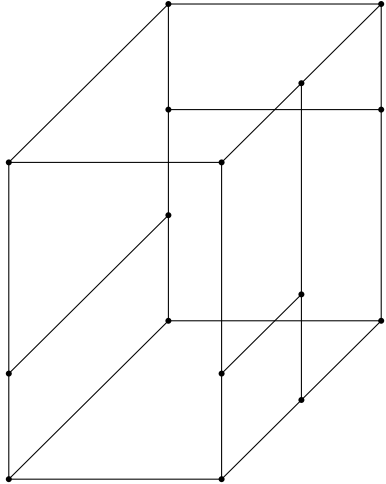
\begin{figure}[ht]
\centering
\begin{tikzpicture}[scale=0.7, line cap=round, line join=round]
  \def\c{2}
  \def\xa{3}
  \def\ya{3}
  \def\l{4}
  \coordinate (q0)   at (0,0);
  \coordinate (qb)   at (\l,0);
  \coordinate (qa)   at (\xa,\ya);
  \coordinate (q1)   at (\xa+\l,\ya);
  \coordinate (qcpt) at ($(qb)!0.5!(q1)$);

  \coordinate (0_0) at ($(q0)+(0,0)$);
  \coordinate (0_1) at ($(q0)+(0,\c)$);
  \coordinate (0_2) at ($(q0)+(0,{3*\c})$);

  \coordinate (b_0) at ($(qb)+(0,0)$);
  \coordinate (b_1) at ($(qb)+(0,\c)$);
  \coordinate (b_2) at ($(qb)+(0,{3*\c})$);

  \coordinate (c_0) at ($(qcpt)+(0,0)$);
  \coordinate (c_1) at ($(qcpt)+(0,\c)$);
  \coordinate (c_2) at ($(qcpt)+(0,{3*\c})$);

  \coordinate (1_0) at ($(q1)+(0,0)$);
  \coordinate (1_1) at ($(q1)+(0,{2*\c})$);
  \coordinate (1_2) at ($(q1)+(0,{3*\c})$);

  \coordinate (a_0) at ($(qa)+(0,0)$);
  \coordinate (a_1) at ($(qa)+(0,\c)$);
  \coordinate (a_2) at ($(qa)+(0,{2*\c})$);
  \coordinate (a_3) at ($(qa)+(0,{3*\c})$);

  \draw (0_0) -- (0_1) -- (0_2);
  \draw (b_0) -- (b_1) -- (b_2);
  \draw (c_0) -- (c_1) -- (c_2);
  \draw (1_0) -- (1_1) -- (1_2);
  \draw (a_0) -- (a_1) -- (a_2) -- (a_3);

  \draw (0_0) -- (a_0) -- (1_0);
  \draw (0_0) -- (b_0) -- (c_0) -- (1_0);

  \draw (0_2) -- (a_3) -- (1_2);
  \draw (0_2) -- (b_2) -- (c_2) -- (1_2);

  \draw (0_1) -- (a_1);
  \draw (b_1) -- (c_1);
  \draw (a_2) -- (1_1);

  \fill (0_0) circle (1.6pt);
  \fill (0_1) circle (1.6pt);
  \fill (0_2) circle (1.6pt);
  \fill (b_0) circle (1.6pt);
  \fill (b_1) circle (1.6pt);
  \fill (b_2) circle (1.6pt);
  \fill (c_0) circle (1.6pt);
  \fill (c_1) circle (1.6pt);
  \fill (c_2) circle (1.6pt);
  \fill (1_0) circle (1.6pt);
  \fill (1_1) circle (1.6pt);
  \fill (1_2) circle (1.6pt);
  \fill (a_0) circle (1.6pt);
  \fill (a_1) circle (1.6pt);
  \fill (a_2) circle (1.6pt);
  \fill (a_3) circle (1.6pt);
\end{tikzpicture}
\caption{A cubic construction which does not admit an order-embedding deformation since its augmented pre-Reeb graph is not acyclic.}
\end{figure}

Now let
\[
\mathrm{pt} = P_0 \xleftarrow{\pi_1} P_1 \xleftarrow{\pi_2} P_2 \xleftarrow{} \cdots
\]
be a tower of cylindrical projections starting from a point, and assume that each augmented pre-Reeb graph \(\widehat R(\pi_n)\) is acyclic. Then Proposition~\ref{prop:good-cubic-extension} constructs inductively an order-embedding cubic realization of each \(P_n\) by adding one coordinate at a time.

At the \(n\)-th step, the possible choices for the new coordinate are exactly the increasing functions on the augmented Reeb poset \(\widehat{\mathsf{Reeb}}(\pi_n)\). Thus the tower of augmented Reeb posets controls all order-embedding cubic realizations compatible with the tower of projections.

If, for every \(n\), the augmented Reeb poset \(\widehat{\mathsf{Reeb}}(\pi_n)\) is a total order, then the resulting cubic realization is combinatorially unique: for any compatible realization \(c\), the relation
\[
c(x)_i\le c(y)_i
\]
does not depend on the choice of \(c\). In particular, one may consider the minimal \(\mathbb Z_{\ge 0}\)-realization, obtained by choosing at each step the height function given by the position in the total order.

\begin{proposition}\label{lem:dimension-cylindrical}
Let
\[
\mathrm{pt} = P_0 \xleftarrow{\pi_1} P_1 \xleftarrow{\pi_2} P_2 \xleftarrow{} \cdots \xleftarrow{\pi_n} P_n
\]
be a tower of cylindrical projections. If each augmented pre-Reeb graph \(\widehat R(\pi_i)\) is acyclic, then
\[
\dim(P_n)=n.
\]
\end{proposition}

\begin{proof}
By Proposition~\ref{prop:good-cubic-extension}, \(P_n\) admits an order-embedding cubic realization in \(\mathbb R^n\). Hence \(\dim(P_n)\le n\).

For each \(i\), let
\[
b_i,t_i:P_{i-1}\hookrightarrow P_i
\]
be the bottom and top sections, and let \(X_n\subset P_n\) be the set of all composites obtained by choosing at each stage either \(b_i\) or \(t_i\). We claim that, in the inductive realization of \(P_n\), the set \(X_n\) is the set of corners of an \(n\)-box.

We argue by induction on \(n\). For \(n=0\), this is clear. Assume it for \(n-1\). Write the inductive realization of \(P_n\) as
\[
c_n(x)=\bigl(c_{n-1}(\pi_n(x)),\,h_n([x])\bigr),
\]
where \(h_n\) is a strictly increasing function on the augmented Reeb poset of \(\pi_n\).

By condition~3) in the definition of cylindrical projection, the images \(b_n(P_{n-1})\) and \(t_n(P_{n-1})\) are horizontal copies of the Hasse diagram of \(P_{n-1}\), hence horizontal classes. They are respectively minimal and maximal in the augmented Reeb poset, so \(h_n\) takes some values \( a_n<b_n\) on them. Therefore, for every \(x\in P_{n-1}\),
\[
c_n(b_n(x))=(c_{n-1}(x),a_n),\qquad
c_n(t_n(x))=(c_{n-1}(x),b_n).
\]

By the induction hypothesis, the points of \(X_{n-1}\) are the corners of an \((n-1)\)-box. Passing from \(X_{n-1}\) to \(X_n\) adds one independent choice for the last coordinate, namely \(a_n\) or \(b_n\). Thus \(X_n\) is the set of corners of an \(n\)-box.

Hence \(X_n\) is a subposet isomorphic to the Boolean lattice \(B_n\). Therefore \(\dim(P_n)\ge \dim(B_n) = n\). Thus \(\dim(P_n)=n\).
\end{proof}

\subsection{Weak orders in types \(A\) and \(B\) }

\begin{definition}
Let \(w\) be a word. For letters \(a,b\) in \(w\), write \(a\prec_w b\) if \(a\) occurs strictly to the left of \(b\) in \(w\).
\end{definition}

\begin{definition}
Let \(S_n\) be the symmetric group, written as words on \(\{1,\dots,n\}\).
The (right) \emph{weak order} is the poset whose cover relations are given by swapping adjacent letters \(\cdots a\,b\cdots \mapsto \cdots b\,a\cdots\) whenever \(a<b\).
\end{definition}

\begin{definition}\label{def:typeA-inversions}
Let \(w\in S_n\). For \(1\le i<j\le n\), the symbol \((j,i)\) is an \emph{inversion} of \(w\) if \(j\prec_w i\).
Let \(\Inv(w)\) be the set of all such symbols \((j,i)\).
\end{definition}

The following is standard.

\begin{lemma}\label{lem:typeA-weak-inversions}
For \(u,v\in S_n\),
\[
u\le v \text{ in weak order}
\quad\Longleftrightarrow\quad
\Inv(u)\subseteq \Inv(v).
\]
\end{lemma}

\begin{definition}
Let \(W_n\) be the hyperoctahedral group, written as words in \(\{\pm1,\dots,\pm n\}\) using each absolute value exactly once. Put the signed order
\[
-n<\cdots<-1<1<\cdots<n.
\]
The (right) \emph{weak order} on \(W_n\) is the poset whose covers are:
\begin{enumerate}[label=\arabic*)]
    \item adjacent swaps \(\cdots a\,b\cdots \mapsto \cdots b\,a\cdots\) whenever \(a<b\) in the signed order;
    \item sign flip of the first letter \(a_1\mapsto -a_1\) whenever \(a_1>0\).
\end{enumerate}

\end{definition}

\begin{definition}\label{lem:typeB-two-letter-setup}
Fix \(1\le i<j\le n\). Consider the formal symbols
\[
(-i),\qquad (j,-i),\qquad (j,i).
\]
Let \(w\in W_n\). We define when these symbols belong to \(\Inv(w)\) by the following rules:
\begin{itemize}
\item \((-i)\in\Inv(w)\) iff the letter \(-i\) occurs in \(w\).

\item \((j,-i)\in\Inv(w)\) iff one of the following holds:
\[
(\pm j)\prec_w (-i), \quad (\pm i)\prec_w (-j).
\]

\item \((j,i)\in\Inv(w)\) iff one of the following holds:
\[
(\pm j)\prec_w i, \quad (\pm i)\prec_w (-j).
\]
\end{itemize}
Let \(\Inv(w)\) be the subset of all such symbols \((-i),(j,-i),(j,i)\) (over all \(1\le i<j\le n\)) selected by these rules.
\end{definition}

The following is again standard \cite{BB05}.

\begin{lemma}\label{lem:typeB-weak-inversions}
For \(u,v\in W_n\),
\[
u\le v \text{ in weak order}
\quad\Longleftrightarrow\quad
\Inv(u)\subseteq \Inv(v).
\]

\end{lemma}

\section{Type \(A\): deletion \(S_n\to S_{n-1}\)}

Throughout this section, \(\pi:S_n\to S_{n-1}\) is the deletion map erasing the letter \(n\). This is the projection to the maximal parabolic subgroup \(S_{n-1}\); its fibers are the classes of the corresponding parabolic congruence on weak order, in the sense of Reading \cite{Rea04}.

\begin{lemma}\label{lem:typeA-pi-totally-ordered-fibers}
The map \(\pi\) is a cylindrical projection.
\end{lemma}

\begin{proof}
A cover in the right weak order on \(S_n\) is an adjacent swap
\[
\cdots a\,b\cdots \mapsto \cdots b\,a\cdots
\qquad(a<b).
\]
If neither \(a\) nor \(b\) is \(n\), then deleting \(n\) produces the corresponding cover in \(S_{n-1}\).
If one of \(\{a,b\}\) is \(n\), then deleting \(n\) yields the same word on both sides. Hence \(\pi\) satisfies the cover condition.

Fix \(v=v_1\cdots v_{n-1}\in S_{n-1}\). The fiber \(\pi^{-1}(v)\) consists of the \(n\) words obtained by inserting \(n\) into \(v\), and these form a chain:
\[
v_1\cdots v_{n-1}n \;<\; v_1\cdots v_{n-2}n v_{n-1}\;<\;\cdots\;<\; v_1 n v_2\cdots v_{n-1}\;<\; n v_1\cdots v_{n-1}.
\]
Indeed, each step is the adjacent swap
\[
v_k\,n\mapsto n\,v_k,
\]
which is a cover because \(v_k<n\). Thus every fiber is totally ordered.

The bottom section sends \(v\) to \(v_1\cdots v_{n-1}n\), and the top section sends \(v\) to \(n v_1\cdots v_{n-1}\). Their images are the induced subposets of \(S_n\) consisting respectively of permutations with \(n\) at the end and with \(n\) at the beginning. Both are evidently isomorphic to \(S_{n-1}\), with any cover of  \(S_{n-1}\) present in the respective subposet of \(S_{n}\). Hence \(\pi\) is cylindrical.
\end{proof}

\subsection{Pre-Reeb graph and Reeb poset in type \(A\)}
Let  \(\mathcal P([n-1])\) denote the power set of \([n-1]=\{1,\dots,n-1\}\).
\begin{proposition}\label{prop:typeA-preReeb}
The pre-Reeb graph \(R(\pi)\) has vertex set canonically identified with \(\mathcal P([n-1])\).
Under this identification, the edges are precisely
\[
A \longrightarrow A\cup\{a\}\qquad(a\in [n-1]\setminus A).
\]
In particular, \(R(\pi)\) is the Hasse diagram of the Boolean lattice
\((\mathcal P([n-1]),\subseteq)\), and hence
\[
\mathsf{Reeb}(\pi)\ \cong\ (\mathcal P([n-1]),\subseteq).
\]
\end{proposition}

\begin{proof}
For \(w\in S_n\), define
\[
A(w):=\{\,i\in[n-1]\mid n\prec_w i\,\},
\]
i.e.\ the set of letters \(i\) such that \((n,i) \in \Inv(w)\).

\emph{Step 1: horizontal classes.}
If \(w\to w'\) is a horizontal cover, then the adjacent swap does not involve \(n\), hence no letter crosses \(n\), so \(A(w)=A(w')\).
Conversely, if \(A(w)=A(w')\), write \(w=u\,n\,v\) and \(w'=u'\,n\,v'\) where \(v,v'\) are words on the same letter set \(A(w)\) and \(u,u'\) are words on the complementary letter set \([n-1]\setminus A(w)\).
Using adjacent swaps inside \(u\) and inside \(v\) (never crossing \(n\)), one connects \(w\) to \(w'\) by a zig-zag of horizontal covers. Thus horizontal classes are indexed by subsets \(A\subseteq[n-1]\) via \([w]\leftrightarrow A(w)\).

\emph{Step 2: vertical covers.}
A cover is vertical iff it involves \(n\), hence has the form
\[
\cdots\, a\, n\, \cdots\ \longrightarrow\ \cdots\, n\, a\, \cdots \qquad (a\in[n-1]).
\]
If \(A=A(w)\), then \(a\) is immediately left of \(n\), hence \(a\notin A\), and after the swap \(a\) lies to the right of \(n\), so the new subset is \(A\cup\{a\}\).
Therefore every vertical cover yields an edge \(A\to A\cup\{a\}\) in \(R(\pi)\).

Conversely, given \(A\subseteq[n-1]\) and \(a\notin A\), choose a representative \(w\) in the class \(A\) with \(a\) immediately left of \(n\), i.e.
\[
w=(\text{a word on }[n-1]\setminus(A\cup\{a\}))\,a\,n\,(\text{a word on }A).
\]
Then swapping \(a\,n\mapsto n\,a\) is a vertical cover producing the edge \(A\to A\cup\{a\}\).
Hence these are exactly all edges of \(R(\pi)\).

\emph{Step 3: reachability.}
Along edges \(A\to A\cup\{a\}\) the set increases, so \(R(\pi)\) is acyclic.
A directed path from \(A\) to \(B\) exists iff \(A\subseteq B\), hence the reachability poset is
\((\mathcal P([n-1]),\subseteq)\).
\end{proof}

\subsection{Augmented pre-Reeb graph and augmented Reeb poset in type \(A\)}

We keep the notation of Proposition~\ref{prop:typeA-preReeb}, so \(V(R(\pi))\cong\mathcal P([n-1])\).
Define the binary valuation
\[
\nu(A):=\sum_{i\in A}2^{\,i-1}\in\{0,1,\dots,2^{n-1}-1\}.
\]
For \(A\neq [n-1]\), define its \emph{successor} \(s(A)\) as follows: let
\[
m:=\min\bigl([n-1]\setminus A\bigr),
\qquad
s(A):=\bigl(A\setminus [m-1]\bigr)\cup\{m\}.
\]
Then \(\nu(s(A))=\nu(A)+1\).

\begin{lemma}\label{lem:typeA-nu-decreases}
If \(A\to B\) is an edge of \(\widehat{R}(\pi)\), then \(\nu(B)>\nu(A)\).
\end{lemma}

\begin{proof}
We split into edges of \(R(\pi)\) and auxiliary edges.

\medskip\noindent
\textbf{Step 1: edges of \(R(\pi)\).}
If \(A\to B\) is an edge of \(R(\pi)\), then \(B=A\cup\{a\}\) for some \(a\notin A\). Hence
\[
\nu(B)=\nu(A)+2^{a-1}>\nu(A).
\]

\medskip\noindent
\textbf{Step 2: auxiliary edges.}
Let \(A\to B\) be an auxiliary edge. Fix witnesses \(v\in A\) and \(w\in B\) with
\[
\pi(w)<\pi(v)\quad\text{in }S_{n-1},
\qquad
v,w\text{ incomparable in }S_n.
\]
The binary strings of \(\nu(A)\) and \(\nu(B)\) are exactly the characteristic strings of the
\(n\)-inversions
\[
(n,n-1),\dots,(n,1)
\]
of \(v\) and \(w\): the bit at position \(i\) records whether \(i\in A\) or \(i\in B\), equivalently whether \((n,i)\in\Inv(v)\) or \((n,i)\in\Inv(w)\).

Let \(t\) be the most significant position where these two bitstrings differ. We claim that at \(t\) the bit changes from \(0\) for \(A\) to \(1\) for \(B\), equivalently \(t\in B\setminus A\).

Assume for contradiction that \(t\in A\setminus B\).
 Since \(\pi(w)\le \pi(v)\), every inversion of \(\pi(w)\) is an inversion of \(\pi(v)\). Hence the incomparability of \(v\), \(w\) must come from the \(n\)-inversions, so there exists some $j < t$ such that \(j\in B\setminus A\). 
Now in \(w\), the letter \(t\) lies to the left of \(n\) while \(j\) lies to the right of \(n\), so
\[
t\prec_{\pi(w)} j,
\]
hence \((t,j)\in\Inv(\pi(w))\).
On the other hand, in \(v\), the letter \(j\) lies to the left of \(n\) while \(t\) lies to the right of \(n\), so
\[
j\prec_{\pi(v)} t,
\]
hence \((t,j)\notin\Inv(\pi(v))\).
This contradicts \(\pi(w)\le \pi(v)\).

Therefore the most significant differing bit cannot change from \(1\) to \(0\). So it changes from \(0\) to \(1\), i.e.\ \(t\in B\setminus A\). Hence \(\nu(B)>\nu(A)\).
\end{proof}

By Lemma~\ref{lem:typeA-nu-decreases}, every edge of \(\widehat{R}(\pi)\) strictly increases \(\nu\); hence \(\widehat{R}(\pi)\) is acyclic.

\begin{proposition}\label{prop:typeA-critical-aux}
For every \(A\neq [n-1]\), the edge
\[
A \longrightarrow s(A)
\]
belongs to \(\widehat{R}(\pi)\).
More precisely, it is an edge of \(R(\pi)\) when \(1\notin A\), and it is an auxiliary edge when \(1\in A\).
\end{proposition}

\begin{proof}
Let \(m=\min([n-1]\setminus A)\), so \(s(A)=(A\setminus[m-1])\cup\{m\}\).

If \(m=1\) (equivalently \(1\notin A\)), then \(s(A)=A\cup\{1\}\), so \(A\to s(A)\) is one of the old edges of \(R(\pi)\).

Assume \(m>1\). Write \(\mathrm{inc}(S)\) (resp.\ \(\mathrm{dec}(S)\)) for the word listing \(S\) in increasing (resp.\ decreasing) order, and set \(B:=s(A)\).
Define
\[
v:= \bigl(\mathrm{dec}([n-1]\setminus A)\bigr)\, n \,\bigl(\mathrm{dec}(A)\bigr)\in S_n,
\qquad
w:= \bigl(\mathrm{inc}([n-1]\setminus B)\bigr)\, n \,\bigl(\mathrm{inc}(B)\bigr)\in S_n.
\]
Then \(v\) lies in the horizontal class \(A\) and \(w\) lies in the horizontal class \(B\), and
\[
\pi(v)=\mathrm{dec}([n-1]\setminus A)\,\mathrm{dec}(A),
\qquad
\pi(w)=\mathrm{inc}([n-1]\setminus B)\,\mathrm{inc}(B).
\]
The word \(\pi(w)\) has no inversions inside either block, so its inversions are exactly the cross-block pairs \((j,i)\) with \(i\in B\), \(j\notin B\), and \(i<j\).
Each such pair is also an inversion of \(\pi(v)\), hence \(\mathrm{Inv}(\pi(w))\subseteq \mathrm{Inv}(\pi(v))\), so \(\pi(w)\le \pi(v)\).
This inclusion is strict since \(m\notin A\) but \(m-1\in A\), hence \((m,m-1)\in\mathrm{Inv}(\pi(v))\) while \((m,m-1)\notin\mathrm{Inv}(\pi(w))\).
Thus \(\pi(w)<\pi(v)\).

Finally, \(m\in B\setminus A\) implies \((n,m)\in \mathrm{Inv}(w)\setminus \mathrm{Inv}(v)\), while \([m-1]\subseteq A\setminus B\) implies \((n,j)\in \mathrm{Inv}(v)\setminus \mathrm{Inv}(w)\) for every \(j<m\).
Hence \(v,w\) are incomparable in \(S_n\), so \(A\to B\) is an auxiliary edge, i.e.\ an edge of \(\widehat{R}(\pi)\).
\end{proof}

\begin{corollary}\label{cor:typeA-augreeb-chain}
The augmented Reeb poset \(\widehat{\mathsf{Reeb}}(\pi)\) is the total order on
\(\mathcal P([n-1])\) induced by \(\nu\):
\[
A \le_{\widehat{\mathsf{Reeb}}(\pi)} B \quad\Longleftrightarrow\quad \nu(A)\le \nu(B).
\]
\end{corollary}

\begin{proof}
If \(A\rightsquigarrow B\) is a directed path in \(\widehat{R}(\pi)\), then Lemma~\ref{lem:typeA-nu-decreases} implies \(\nu(A)\le \nu(B)\).

Conversely, assume \(\nu(A)\le \nu(B)\). Iterating Proposition~\ref{prop:typeA-critical-aux} gives a directed path
\[
A\to s(A)\to s^{2}(A)\to\cdots\to s^{\,\nu(B)-\nu(A)}(A),
\]
and \(\nu\bigl(s^{k}(A)\bigr)=\nu(A)+k\). Since \(\nu\) is a bijection \(\mathcal P([n-1])\to\{0,\dots,2^{n-1}-1\}\), the endpoint equals \(B\).
\end{proof}

The total order \(\widehat{\mathsf{Reeb}}(\pi)\) has \(2^{n-1}\) elements, hence the valuation \(\nu\) is the minimal \(\mathbb Z_{\ge 0}\)-valued height function on \(\widehat{\mathsf{Reeb}}(\pi)\). Thus, in type \(A\), the cubic coordinates constructed by taking height \(\nu\) at each step are the minimal \(\mathbb Z_{\ge 0}\)-coordinates compatible with the tower of deletion projections. 

\section{Type \(B\): deletion \(W_n\to W_{n-1}\)}

Throughout this section, \(\pi:W_n\to W_{n-1}\) is the deletion map erasing the letter \(\pm n\).
This is the projection to the maximal parabolic subgroup \(W_{n-1}\); its fibers are the classes of the corresponding parabolic congruence on weak order, in the sense of Reading \cite{Rea04}.

\begin{lemma}\label{lem:typeB-pi-totally-ordered-fibers}
The map \(\pi\) is a cylindrical projection.
\end{lemma}

\begin{proof}
A cover in the right weak order on \(W_n\) is either an adjacent swap
\[
\cdots a\,b\cdots \mapsto \cdots b\,a\cdots
\qquad(a<b)
\]
in the signed order
\[
-n<\cdots<-1<1<\cdots<n,
\]
or a sign flip of the first letter
\[
a_1\mapsto -a_1
\qquad(a_1>0).
\]
If the move does not involve \(\pm n\), deleting \(\pm n\) gives the corresponding cover in \(W_{n-1}\); otherwise deletion gives equality. Hence \(\pi\) satisfies the cover condition.

Fix \(v=v_1\cdots v_{n-1}\in W_{n-1}\). The fiber \(\pi^{-1}(v)\) consists of inserting \(n\) or \(-n\) into \(v\), hence has \(2n\) elements, and these form the chain
\[
v_1\cdots v_{n-1}n \;<\; v_1\cdots v_{n-2}n v_{n-1}\;<\;\cdots\;<\; n v_1\cdots v_{n-1}
\]
\[
\;<\; -n v_1\cdots v_{n-1}\;<\; v_1(-n)v_2\cdots v_{n-1}\;<\;\cdots\;<\; v_1\cdots v_{n-1}(-n).
\]
Indeed, \(a\,n\mapsto n\,a\) is a cover for every neighbor \(a\), \(n\mapsto -n\) in the first position is a cover, and \((-n)\,a\mapsto a\,(-n)\) is a cover for every neighbor \(a\). Thus every fiber is totally ordered.

The bottom section sends \(v\) to \(v_1\cdots v_{n-1}n\), and the top section sends \(v\) to \(v_1\cdots v_{n-1}(-n)\). Their images are the induced subposets of \(W_n\) consisting respectively of signed permutations with \(n\) at the end and with \(-n\) at the end. Both are evidently isomorphic to \(W_{n-1}\), with any cover of  \(W_{n-1}\) present in the respective subposet of \(W_{n}\). Hence \(\pi\) is cylindrical.
\end{proof}

\subsection{Pre-Reeb graph and Reeb poset in type B}

\begin{proposition}\label{prop:typeB-preReeb}\label{prop:typeB-Reeb-graph}
Vertices of \(R(\pi)\) are naturally parametrized by pairs \((\varepsilon,A)\) where
\begin{itemize}
    \item \(\varepsilon\in\{+,-\}\);
    \item \(A\subseteq\{\pm1,\dots,\pm(n-1)\}\) satisfies: for each \(i\in[n-1]\), at most one of \(+i,-i\) lies in \(A\).
\end{itemize}
Write \(|A|:=\{\,|a|\mid a\in A\,\}\subseteq[n-1]\).
In this parametrization, the edges of \(R(\pi)\) are exactly:
\begin{enumerate}[label=\arabic*)]
    \item for \(\varepsilon=+\), for every \(i\in[n-1]\setminus|A|\) and every \(\delta\in\{+,-\}\),
    \[
    (+,A)\ \longrightarrow\ \bigl(+,\,A\cup\{\delta i\}\bigr);
    \]
    \item for \(\varepsilon=-\), for every \(a\in A\),
    \[
    (-,A)\ \longrightarrow\ \bigl(-,\,A\setminus\{a\}\bigr);
    \]
    \item only when \(|A|=[n-1]\),
    \[
    (+,A)\ \longrightarrow\ (-,A).
    \]
\end{enumerate}
\end{proposition}

\begin{proof}
For \(w\in W_n\), let \(\varepsilon(w)\) be the sign of \( \pm n\) occurring in \(w\), and define
\[
A(w):=\{\,a\in\{\pm1,\dots,\pm(n-1)\}\mid \pm n \prec_w a\,\},
\]
the set of signed letters lying to the right of \(\pm n\). 

\emph{Step 1: horizontal classes.}
If \(w\to w'\) is a horizontal cover, the move does not involve \(\pm n\), so no letter crosses it. The only sign flip is at the first position, and a letter to the right of \(\pm n\) cannot reach position \(1\) without crossing \(\pm n \). Hence \(\varepsilon(w)=\varepsilon(w')\) and \(A(w)=A(w')\).

Conversely, if \(w,w'\) have the same pair \((\varepsilon,A)\), write \(w=u\,\varepsilon n\,v\) and \(w'=u'\, \varepsilon n\,v'\), where \(v,v'\) are words on the same signed letter set \(A\), and \(u,u'\) use the complementary absolute values \([n-1]\setminus|A|\).
Adjacent swaps not crossing \(\pm n\) reorder \(u\) and \(v\) independently; moreover any letter on the left of \(\pm n\) can be moved to position \(1\), flipped there, and moved back, without crossing \(\pm n\).
Thus \(w\) and \(w'\) lie in the same horizontal class. Hence horizontal classes are exactly the pairs \((\varepsilon,A)\).

\emph{Step 2: vertical covers.}
A cover is vertical iff it involves the deleted letter \(\pm n\).

If \(\varepsilon=+\), then \(+n\) is maximal, so the only adjacent-swap cover involving \(+n\) swaps it with its left neighbor:
\[
\cdots a\,(+n)\cdots\ \longrightarrow\ \cdots (+n)\,a\cdots.
\]
This moves the signed letter \(a\) from the left to the right of \(+n\), so \(A\) changes to \(A\cup\{a\}\).
Since the sign of a left letter can be chosen freely within its horizontal class, for each \(i\in[n-1]\setminus|A|\) and \(\delta\in\{+,-\}\) we obtain the edge
\((+,A)\to(+,A\cup\{\delta i\})\) as in \(1)\), and conversely each such edge is realized by choosing a representative with \(\delta i\) immediately left of \(+n\).

If \(\varepsilon=-\), then \(-n\) is minimal, so the only adjacent-swap cover involving \(-n\) swaps it with its right neighbor:
\[
\cdots (-n)\,a\cdots\ \longrightarrow\ \cdots a\,(-n)\cdots.
\]
This moves the signed letter \(a\) from the right to the left of \(-n\), hence \(A\) changes to \(A\setminus\{a\}\), giving the edges in \(2)\).

Finally, the sign-flip cover at position \(1\) is vertical exactly when the first letter is \(+n\), i.e. when there are no letters to the left of \(+n\), equivalently \(|A|=[n-1]\); it gives the edge in \(3)\).

These are all vertical covers, hence all edges of \(R(\pi)\).
\end{proof}

We now give an alternative description of \(R(\pi)\) in terms of a zonotope (see \cite{Ziegler1995} for generalities on zonotopes). Let \(F_n\) be the undirected graph with vertex set
\[
V(F_n)=\{L,R\}\sqcup [n-1]
\]
and edge set
\[
E(F_n)=\{LR\}\ \sqcup\ \{L i : i\in[n-1]\}\ \sqcup\ \{iR : i\in[n-1]\}.
\]
Let \(\mathsf{AO}(F_n)\) be the set of acyclic orientations of \(F_n\), and let \(\mathcal O_0\) be the orientation
\[
L\to R,\qquad L\to i,\qquad i\to R \qquad (i\in[n-1]).
\]

\begin{definition}\label{def:GammaFn}
Define a directed graph \(\Gamma(F_n)\) as follows. Its vertex set is \(\mathsf{AO}(F_n)\).
There is a directed edge \(\mathcal O\to \mathcal O'\) if \(\mathcal O'\) is obtained from \(\mathcal O\) by reversing
exactly one edge of \(F_n\), and that edge is directed in \(\mathcal O\) as in \(\mathcal O_0\).
\end{definition}

A single flip \(\mathcal O\to \mathcal O'\) thus increases by \(1\) the number of edges on which the orientation disagrees with \(\mathcal O_0\), so \(\Gamma(F_n)\) is acyclic. 


\newcommand{\Fthree}[6]{%
  \tikz[baseline=-0.65ex,scale=0.42,>=Stealth,
        lbl/.style={inner sep=0pt,outer sep=0pt,font=\scriptsize},
        edge/.style={line width=0.40pt,shorten <=0.8pt,shorten >=0.8pt}]%
  {%
    \node[lbl] (L)   at (0,0)        {$L$};
    \node[lbl] (R)   at (1.35,0)     {$R$};
    \node[lbl] (One) at (0.675,0.55) {$1$};
    \node[lbl] (Two) at (0.675,1.15) {$2$};

    \draw[edge,#1] (L) -- (R);
    \draw[edge,#2] (L) -- (One);
    \draw[edge,#3] (One) -- (R);
    \draw[edge,#4] (L) -- (Two);
    \draw[edge,#5] (Two) -- (R);

    \node[font=\tiny,inner sep=0pt,outer sep=0pt] at (0.675,-0.32) {$#6$};
  }%
}

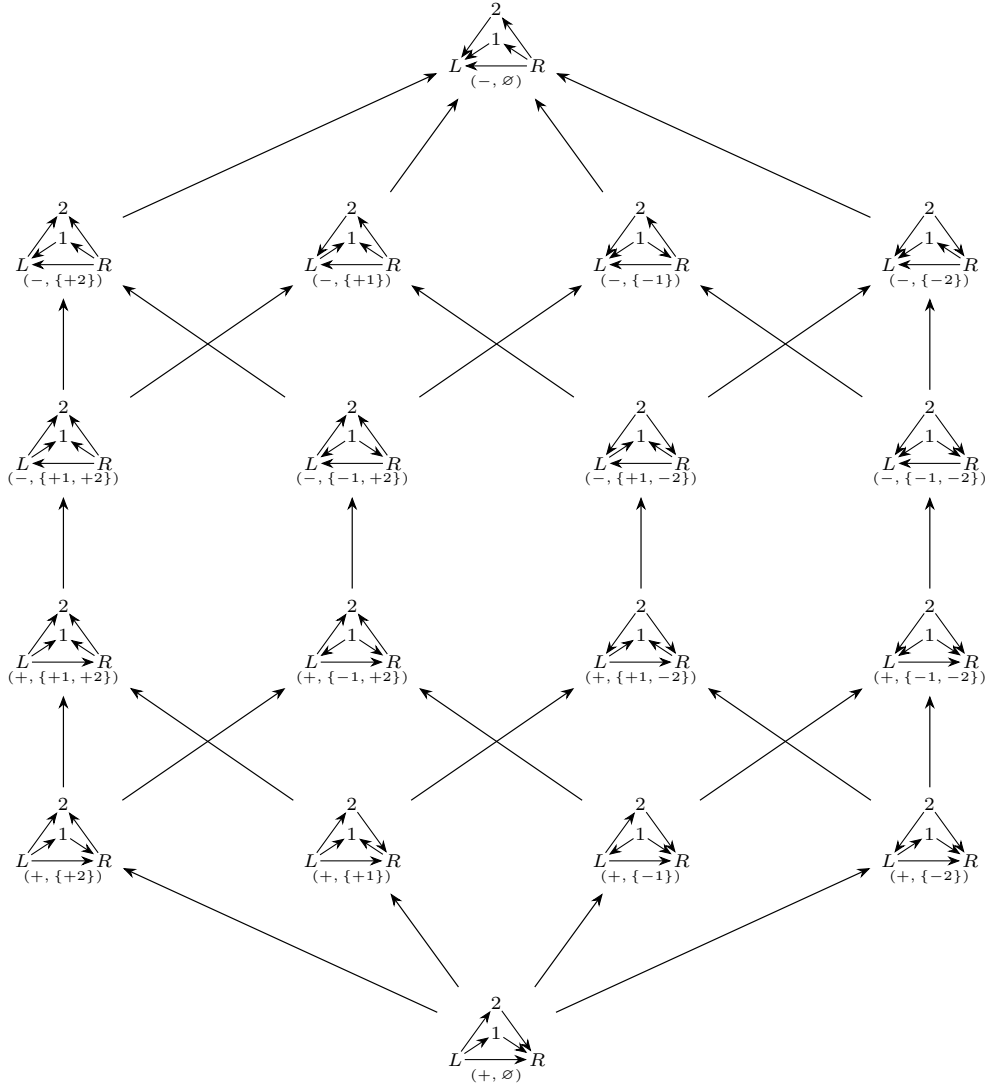
\begin{figure}[t]
\centering
\begin{tikzpicture}[x=1.9cm,y=1.55cm,>=Stealth,
  vertex/.style={inner sep=2pt},
  edge/.style={->,line width=0.45pt,shorten <=2pt,shorten >=2pt}]

\node[vertex] (r0)  at (0,0)    {\Fthree{->}{->}{->}{->}{->}{(+,\varnothing)}};   

\node[vertex] (r11) at (-3,1.7)   {\Fthree{->}{->}{->}{->}{<-}{(+,\{+2\})}};
\node[vertex] (r12) at (-1,1.7)   {\Fthree{->}{->}{<-}{->}{->}{(+,\{+1\})}};
\node[vertex] (r13) at (1,1.7)    {\Fthree{->}{<-}{->}{->}{->}{(+,\{-1\})}};
\node[vertex] (r14) at (3,1.7)    {\Fthree{->}{->}{->}{<-}{->}{(+,\{-2\})}};

\node[vertex] (r21) at (-3,1.7*2)   {\Fthree{->}{->}{<-}{->}{<-}{(+,\{+1,+2\})}};
\node[vertex] (r22) at (-1,1.7*2)   {\Fthree{->}{<-}{->}{->}{<-}{(+,\{-1,+2\})}};
\node[vertex] (r23) at (1,1.7*2)    {\Fthree{->}{->}{<-}{<-}{->}{(+,\{+1,-2\})}};
\node[vertex] (r24) at (3,1.7*2)    {\Fthree{->}{<-}{->}{<-}{->}{(+,\{-1,-2\})}};

\node[vertex] (r31) at (-3,1.7*3)   {\Fthree{<-}{->}{<-}{->}{<-}{(-,\{+1,+2\})}};
\node[vertex] (r32) at (-1,1.7*3)   {\Fthree{<-}{<-}{->}{->}{<-}{(-,\{-1,+2\})}};
\node[vertex] (r33) at (1,1.7*3)    {\Fthree{<-}{->}{<-}{<-}{->}{(-,\{+1,-2\})}};
\node[vertex] (r34) at (3,1.7*3)    {\Fthree{<-}{<-}{->}{<-}{->}{(-,\{-1,-2\})}};

\node[vertex] (r41) at (-3,1.7*4)   {\Fthree{<-}{<-}{<-}{->}{<-}{(-,\{+2\})}};
\node[vertex] (r42) at (-1,1.7*4)   {\Fthree{<-}{->}{<-}{<-}{<-}{(-,\{+1\})}};
\node[vertex] (r43) at (1,1.7*4)    {\Fthree{<-}{<-}{->}{<-}{<-}{(-,\{-1\})}};
\node[vertex] (r44) at (3,1.7*4)    {\Fthree{<-}{<-}{<-}{<-}{->}{(-,\{-2\})}};

\node[vertex] (r5)  at (0,1.7*5)    {\Fthree{<-}{<-}{<-}{<-}{<-}{(-,\varnothing)}};

\draw[edge] (r0) -- (r11);
\draw[edge] (r0) -- (r12);
\draw[edge] (r0) -- (r13);
\draw[edge] (r0) -- (r14);

\draw[edge] (r11) -- (r21);
\draw[edge] (r11) -- (r22);

\draw[edge] (r12) -- (r21);
\draw[edge] (r12) -- (r23);

\draw[edge] (r13) -- (r22);
\draw[edge] (r13) -- (r24);

\draw[edge] (r14) -- (r23);
\draw[edge] (r14) -- (r24);

\draw[edge] (r21) -- (r31);
\draw[edge] (r22) -- (r32);
\draw[edge] (r23) -- (r33);
\draw[edge] (r24) -- (r34);

\draw[edge] (r31) -- (r41);
\draw[edge] (r31) -- (r42);

\draw[edge] (r32) -- (r41);
\draw[edge] (r32) -- (r43);

\draw[edge] (r33) -- (r42);
\draw[edge] (r33) -- (r44);

\draw[edge] (r34) -- (r43);
\draw[edge] (r34) -- (r44);

\draw[edge] (r41) -- (r5);
\draw[edge] (r42) -- (r5);
\draw[edge] (r43) -- (r5);
\draw[edge] (r44) -- (r5);

\end{tikzpicture}
\caption{The directed flip graph \(\Gamma(F_3)\cong R(\pi:W_3\to W_2)\).}
\end{figure}

\begin{proposition}\label{prop:typeB-graphical}
For the deletion projection \(\pi:W_n\to W_{n-1}\), the pre-Reeb graph \(R(\pi)\) is isomorphic, as a directed graph,
to \(\Gamma(F_n)\).
In particular, \(R(\pi)\) is acyclic, hence \(\mathsf{Reeb}(\pi)\) is defined as the reachability poset of \(R(\pi)\).
\end{proposition}

\begin{proof}
\emph{Step 1: acyclic orientations of \(F_n\).}
Let \(\mathcal O\) be an orientation of \(F_n\), and write \(T\to H\) for the oriented edge \(LR\)
(so \(\{T,H\}=\{L,R\}\)).
Then \(\mathcal O\) is acyclic if and only if there is no \(i\in[n-1]\) with \(H\to i\to T\). Hence, for each \(i\), exactly one of the following three patterns occurs:
\[
T\to i\to H,\qquad (T\to i\ \text{and}\ H\to i),\qquad (i\to T\ \text{and}\ i\to H).
\]

\emph{Step 2: a bijection on vertices.}
Recall from Proposition~\ref{prop:typeB-preReeb} that \(V(R(\pi))\) is parametrized by pairs \((\varepsilon,A)\).
Define a map
\[
\Phi:V(R(\pi))\longrightarrow \mathsf{AO}(F_n)
\]
by the following rules.
Orient the edge \(LR\) by
\[
\varepsilon=+\ \Rightarrow\ L\to R,\qquad \varepsilon=-\ \Rightarrow\ R\to L,
\]
and write \(T\to H\) for this oriented edge.
For each \(i\in[n-1]\), set:
\begin{itemize}
    \item if \(i\notin |A|\), orient \(T\to i\to H\);
    \item if \(+i\in A\), orient \(T\to i\) and \(H\to i\);
    \item if \(-i\in A\), orient \(i\to T\) and \(i\to H\).
\end{itemize}
This produces an acyclic orientation by Step~1.
Conversely, given \(\mathcal O\in\mathsf{AO}(F_n)\), recover \(\varepsilon\) from the direction of \(LR\), and recover \(A\) by
\[
T\to i\to H \ \Rightarrow\ \pm i\notin A,\qquad
(T\to i\ \text{and}\ H\to i)\ \Rightarrow\ +i\in A,\qquad
(i\to T\ \text{and}\ i\to H)\ \Rightarrow\ -i\in A.
\]
This is inverse to \(\Phi\), so \(\Phi\) is a bijection.

\emph{Step 3: compatibility with directed edges.}
We show that \(\Phi\) identifies the directed edges of \(R(\pi)\) with the directed flips of Definition~\ref{def:GammaFn}, i.e.\ with single flips away from \(\mathcal O_0\).

If \(x=(+,A)\), then \(LR\) is \(L\to R\). For \(i\notin|A|\), the orientation has \(L\to i\to R\), and the edges of \(R(\pi)\) are
\[
(+,A)\to(+,A\cup\{\delta i\})\qquad (\delta\in\{+,-\}).
\]
If \(\delta=+\), this reverses \(iR\); if \(\delta=-\), this reverses \(Li\). In both cases the reversed edge is originally directed as in \(\mathcal O_0\).

If \(x=(-,A)\), then \(LR\) is \(R\to L\), and the edges of \(R(\pi)\) are
\[
(-,A)\to(-,A\setminus\{a\})\qquad (a\in A).
\]
If \(a=+i\), then \(i\) is a sink and the flip reverses \(Li\); if \(a=-i\), then \(i\) is a source and the flip reverses \(iR\). Again the reversed edge is originally directed as in \(\mathcal O_0\).

Finally, the edge
\[
(+,A)\to(-,A)
\]
occurs exactly when \(|A|=[n-1]\), and then it reverses \(LR\), which is also originally directed as in \(\mathcal O_0\).

Thus \(\Phi\) is a directed-graph isomorphism \(R(\pi)\cong \Gamma(F_n)\).
\end{proof}

\begin{remark}\label{rem:zonotope}
The underlying undirected graph of \(\Gamma(F_n)\) is the \(1\)-skeleton of the graphical zonotope \(Z(F_n)\). Equivalently, \(R(\pi)\) identifies with a natural acyclic orientation of the \(1\)-skeleton of \(Z(F_n)\) by a linear functional. Yet equivalently, $R(\pi)$ is the graph of chambers of a hyperplane arrangement, oriented away from a base chamber.
\end{remark}

\subsection{Augmented pre-Reeb graph and augmented Reeb poset in type \(B\)}

Throughout this subsection, \(\pi:W_n\to W_{n-1}\) is the deletion map, and vertices of \(R(\pi)\) are the horizontal classes
\(x=(\varepsilon,A)\) from Proposition~\ref{prop:typeB-preReeb}.

\begin{definition}\label{def:typeB-bits-valuation}\label{def:typeB-valuation}\label{def:valuation}
For \(x=(\varepsilon,A)\), write \(|A|:=\{\,|a|\mid a\in A\,\}\subseteq[n-1]\).

Define the bit
\[
(-n)(x):=\mathbf{1}_{\varepsilon=-}.
\]
For each \(i\in[n-1]\) define bits \((n,i)(x),(n,-i)(x)\in\{0,1\}\) by
\[
(n,i)(x):=\mathbf{1}_{\, +i\in A}+\mathbf{1}_{\,\varepsilon=-,\ i\notin |A|},
\qquad
(n,-i)(x):=\mathbf{1}_{\, -i\in A}+\mathbf{1}_{\,\varepsilon=-,\ i\notin |A|}.
\]
(These sums are disjoint, hence take values in \(\{0,1\}\).)

If \(w\) is any representative of the vertex \(x\), then
\[
(-n)(x)=\mathbf 1_{\,(-n)\in\Inv(w)},\qquad
(n,i)(x)=\mathbf 1_{\, (n,i)\in\Inv(w)},\qquad
(n,-i)(x)=\mathbf 1_{\, (n,-i)\in\Inv(w)}.
\]

Define the valuation \( \nu \) on vertices by
\[
\nu (x):=\sum_{i=1}^{n-1} (n,-i)(x)\,2^{\,n-1-i}\;+\;(-n)(x)\,2^{\,n-1}\;+\;
\sum_{i=1}^{n-1} (n,i)(x)\,2^{\,n-1+i}.
\]
Equivalently, \(\nu (x)\) is the integer represented by the binary string
\[
(n,n\!-\!1)(x)\cdots(n,2)(x)\ (n,1)(x)\ (-n)(x)\ (n,-1)(x)\ (n,-2)(x)\cdots(n,-(n\!-\!1))(x).
\]
\end{definition}

\begin{definition}\label{def:typeB-successor}\label{def:successor}
The successor map \(s\) is a partial function on vertices, defined by the following mutually
exclusive cases.

\medskip\noindent
\textbf{A. The case \(\varepsilon=+\).}
\begin{itemize}
\item[\textbf{A1.}] If \(|A|=[n-1]\), set
\[
s(+,A)=(-,A).
\]
\item[\textbf{A2.}] If \(|A|\neq[n-1]\), let
\[
m:=\max\bigl([n-1]\setminus |A|\bigr).
\]
Set
\[
s(+,A)=\Bigl(+,\,\bigl(A\setminus\{-i:\ i > m\}\bigr)\cup\{-m\}\Bigr).
\]
\end{itemize}

\medskip\noindent
\textbf{B. The case \(\varepsilon=-\).}
Let
\[
P:=\{\,i\in |A|:\ +i\in A\,\}.
\]
\begin{itemize}
\item[\textbf{B1.}] If \(P\neq\varnothing\), let
\[
k=\max(P).
\]
Set
\[
s(-,A)=\Bigl(-,\,(A\setminus\{+k\})\cup\{+i:\ i>k,\ i\notin |A|\}\Bigr).
\]

\item[\textbf{B2.}] If \(P=\varnothing\): if \(A=\varnothing\), then \(s\) is undefined. If \(A\neq\varnothing\), let
\[
s_0:=\min(|A|),
\qquad
A':=\{+s_0\}\ \cup\ \{\, +i:\ i>s_0,\ i\notin |A|\,\}.
\]
Then set
\[
s(-,A)=(+,A').
\]
\end{itemize}
\end{definition}

By direct inspection of the cases in Definition~\ref{def:successor}, the predecessor \(x\) can be uniquely recovered from \(s(x)\) whenever \(s(x)\) is defined. Thus every vertex except \(x_{\min}:=(+,\varnothing)\) has a unique predecessor, and every vertex except \(x_{\max}:=(-,\varnothing)\) has a unique successor. As we will now show in Proposition~\ref{prop:succ_increases_valuation}, the valuation \(\nu\) strictly increases along successor edges, so the directed graph on the vertices with edges \(x\to s(x)\) has no directed cycles. It follows that this graph is a single directed path passing through all vertices exactly once.

\begin{proposition}\label{prop:typeB-successor-increases}\label{prop:succ_increases_valuation}
If \(s(x)\) is defined then \(\nu (s(x))>\nu (x)\).
\end{proposition}

\begin{proof}
Write \(y=s(x)\). In each case we identify the leftmost bit in the valuation string
\[
(n,n\!-\!1)\cdots(n,1)\ (-n)\ (n,-1)\cdots(n,-(n\!-\!1))
\]
that changes from \(x\) to \(y\), and check that it flips \(0\to1\).

\medskip\noindent
\textbf{Case A: \(\varepsilon=+\).}
\begin{itemize}
\item[\textbf{A1.}] \(|A|=[n-1]\). Only the bit \((-n)\) changes, from \(0\) to \(1\); all \((n,i)\) and \((n,-i)\) are unchanged. Hence \(\nu(y)>\nu(x)\).

\item[\textbf{A2.}] \(|A|\neq[n-1]\), and \(m=\max([n-1]\setminus|A|)\). Then \((-n)\) and all \((n,i)\) are unchanged. Also \((n,-m)(x)=0\) and \((n,-m)(y)=1\), since \(x\) has neither \(\pm m\) while \(y\) contains \(-m\). For each \(i>m\) with \(-i\in A\), the passage \(x\mapsto y\) removes \(-i\), so \((n,-i)\) flips \(1\to0\); these bits are less significant than \((n,-m)\). Thus the leftmost changed bit is \((n,-m)\), and it flips \(0\to1\). Hence \(\nu(y)>\nu(x)\).
\end{itemize}

\medskip\noindent
\textbf{Case B: \(\varepsilon=-\).}
\begin{itemize}
\item[\textbf{B1.}] \(P\neq\varnothing\), and \(k=\max(P)\). All \((n,i)\) bits are unchanged. Since \(+k\in A\), we have \((n,-k)(x)=0\); after removing \(+k\) from \(A\) we get \((n,-k)(y)=1\). For each \(i>k\) with \(i\notin|A|\), the passage from \(x\) to \(y\) adds \(+i\) to \(A\), so \((n,-i)\) flips \(1\to0\); these bits are less significant than \((n,-k)\). Thus the leftmost changed bit is \((n,-k)\), and it flips \(0\to1\). Hence \(\nu(y)>\nu(x)\).

\item[\textbf{B2.}] \(P=\varnothing\) and \(A\neq\varnothing\). Let \(s_0=\min(|A|)\). Then \(y\) has \(\varepsilon=+\). We claim
\[
(n,i)(y)=(n,i)(x)\ \text{for all }i>s_0,\qquad (n,s_0)(x)=0,\ (n,s_0)(y)=1.
\]
Indeed, if \(i>s_0\) and \(i\notin|A|\), then \((n,i)(x)=1\) and \(+i\in A_y\), so \((n,i)(y)=1\). If \(i>s_0\) and \(i\in|A|\), then \(P=\varnothing\) forces \(-i\in A\), so \((n,i)(x)=0\), and \(i\notin|A_y|\), so \((n,i)(y)=0\). Finally, \(-s_0\in A\) gives \((n,s_0)(x)=0\), while \(+s_0\in A_y\) gives \((n,s_0)(y)=1\). Thus the leftmost changed bit is \((n,s_0)\), and it flips \(0\to1\). Hence \(\nu(y)>\nu(x)\).
\end{itemize}
\end{proof}

The next lemma follows from the work of Reading \cite{Reading2003}. We give an explicit proof of it, to keep the current paper self-contained.

\begin{lemma}\label{lem:typeB-Lemma2.2}
If \(x\to y\) is an edge of \(\widehat R(\pi)\), then \(\nu(y)>\nu(x)\).
\end{lemma}

\begin{proof}
We split into edges of \(R(\pi)\) and auxiliary edges.

\medskip\noindent
\textbf{Step 1: edges of \(R(\pi)\).}
By the paragraph after Definition~\ref{def:typeB-valuation}, the bitstring of a vertex is exactly the bitstring of \(n\)-inversions of any representative:
\[
(n,n\!-\!1),\dots,(n,1),(-n),(n,-1),\dots,(n,-(n\!-\!1)).
\]
A vertical cover is of one of two kinds: either it swaps \(\pm n\) with one adjacent letter, in which case it changes exactly one \(n\)-inversion and no other \(n\)-inversions; or it flips \(+n\) to \(-n\) at the first position, in which case it changes only the inversion \((-n)\). Therefore every edge of \(R(\pi)\) strictly increases \(\nu\).

\medskip\noindent
\textbf{Step 2: auxiliary edges.}
Let \(x\to y\) be an auxiliary edge. Fix witnesses \(u\in x\) and \(w\in y\) with
\(\pi(w)<\pi(u)\) in \(W_{n-1}\) and \(u,w\) incomparable in \(W_n\).
The bitstrings of \(x\) and \(y\) are the bitstrings of \(n\)-inversions of \(u\) and \(w\). Let \(\lambda\) be the most significant position where these bitstrings differ.

We show that at \(\lambda\) the entry changes from \(0\) for \(u\) to \(1\) for \(w\).
Assume for contradiction that at \(\lambda\) it changes from \(1\) for \(u\) to \(0\) for \(w\).
Since \(\pi(w)\le\pi(u)\), every inversion of \(\pi(w)\) is an inversion of \(\pi(u)\). Hence the incomparability of \(u,w\) must come from the \(n\)-inversions, so there exists a later position \(\mu\) where the entry changes from \(0\) for \(u\) to \(1\) for \(w\). Let \(\mu\) be the most significant such position.

\medskip\noindent
\textbf{Case A: \(\lambda=(n,-i)\).}
Then \(\mu=(n,-j)\) for some \(j>i\). Since \(\lambda\) lies in the \((n,-\cdot)\)-block, all \((n,k)\) and \((-n)\) agree for \(u\) and \(w\).

\[
\begin{array}{c|cc}
 & \lambda=(n,-i) & \mu=(n,-j)\\
\hline
u & 1 & 0\\
w & 0 & 1
\end{array}
\]

If both \(u\) and \(w\) contain \(+n\), then the values of \(\lambda\) and \(\mu\) give:
\begin{itemize}
\item for \(u\), \( (+n \prec_u -i) \) and \( ( \pm j \prec_u +n \text{ or } +n \prec_u +j )\);
\item for \(w\), \( ( \pm i \prec_w +n \text{ or } +n \prec_w +i )\) and \( (+n \prec_w -j) \).
\end{itemize}

We know all more significant bits agree, so \( +n \prec_u -i \implies (n,i) \notin \Inv(u) \implies (n,i) \notin \Inv(w) \), thus \( +n \prec_w +i\) is impossible and we have \( \pm i \prec_w +n \). Hence \( \pm i \prec_w - j \) and \((j,i) \in \Inv(\pi(w))\).
Similarly, \( +n \prec_w -j \implies (n,j) \notin \Inv(w) \implies (n,j) \notin \Inv(u) \), thus \( +n \prec_u +j\) is impossible and we have \( \pm j \prec_u +n \). Hence \( \pm j \prec_u -i\), and \((j,i) \notin \Inv(\pi(u))\), contradicting \(\pi(w) < \pi(u)\). 

\smallskip
If both \(u\) and \(w\) contain \(-n\), then the values of \(\lambda\) and \(\mu\) give:
\begin{itemize}
\item for \(u\), \( ( \pm i \prec_u -n \text{ or } -n \prec_u -i )\) and \( ( -n \prec_u +j )\);
\item for \(w\), \( ( -n \prec_w +i )\) and \( ( \pm j \prec_w -n \text{ or } -n \prec_w -j )\).

We know all more significant bits agree, so \( -n \prec_u +j \implies (n,j) \in \Inv(u) \implies (n,j) \in \Inv(w) \), thus \( -n \prec_w -j\) is impossible and we have \( \pm j \prec_w -n \). Hence \( \pm j \prec_w +i \), and \((j,i) \in \Inv(\pi(w))\). Similarly, \( -n \prec_w +i \implies (n,i) \in \Inv(w) \implies (n,i) \in \Inv(u) \), thus \( -n \prec_u -i\) is impossible and we have \( \pm i \prec_u -n \). Hence \( \pm i \prec_u +j \), and \((j,i) \notin \Inv(\pi(u))\), contradicting \(\pi(w) < \pi(u)\).
\end{itemize}

\medskip\noindent
\textbf{Case B: \(\lambda=(-n)\).}
Then \(\mu=(n,-j)\). 
\[
\begin{array}{c|cc}
 & \lambda=(-n) & \mu=(n,-j)\\
\hline
u & 1 & 0\\
w & 0 & 1
\end{array}
\]

Here \(u\) contains \(-n\) and \(w\) contains \(+n\). The condition \((n,-j)\in\Inv(w)\) gives \(+n\prec_w -j \), while \((n,-j)\notin\Inv(u)\) gives \(-n\prec_u +j\). Hence \((-j)\in\Inv(\pi(w))\) but \((-j)\notin\Inv(\pi(u))\), contradicting \(\pi(w) < \pi(u)\).

\medskip\noindent
\textbf{Case C: \(\lambda=(n,j)\).}
We distinguish three possibilities for \(\mu\).

\begin{itemize}
\item[\textbf{C1.}] \( \mu = (n,i)\), for some \( i<j \).

\[
\begin{array}{c|cc}
 & \lambda=(n,j) & \mu=(n,i)\\
\hline
u & 1 & 0\\
w & 0 & 1
\end{array}
\]

We analyze possible patterns in \(w\), depending on sign of \(n\).

\begin{itemize}
    \item If \( w \) contains \(+n \), then \( (\pm j \prec_w +n \text{ or } +n \prec_w -j )\) and \( (+n \prec_w +i )\). Thus either \(\pm j \prec_w +i\), or \(w\) contains \(+i\) and \(-j\), and in both cases \((j,i) \in \Inv(\pi(w))\).
    \item If \( w \) contains \( -n\), then \( (-n \prec_w -j) \) and \( (  \pm i \prec_w -n  \text{ or }-n \prec_w +i)\). Thus either \(\pm i \prec_w -j\), or \(w\) contains \(+i\) and \(-j\), and still in both cases \( (j,i) \in \Inv(\pi(w))\).
\end{itemize}

We now analyze possible patterns in \(u\), depending on sign of \(n\).

\begin{itemize}
     \item If \( u \) contains \(+n \), then \((+n \prec_u +j)\) and \( ( \pm i \prec_u +n \text{ or } +n \prec_u -i )\). Thus either \( \pm i \prec_u +j \), or \(u\) contains \(+j\) and \(-i\), and in both cases \((j,i) \notin \Inv(\pi(u))\).
    \item If \( u \) contains \( -n\), then \( (\pm j \prec_u -n  \text{ or } -n \prec_u +j )\) and \( (-n \prec_u -i )\). Thus either \(\pm j \prec_u -i\), or \(u\) contains \(+j\) and \(-i\), and still in both cases \((j,i) \notin \Inv(\pi(u))\).
\end{itemize}

So \((j,i) \in \Inv(\pi(w)) \) and \((j,i) \notin \Inv(\pi(u)) \), contradicting \(\pi(w) < \pi(u)\).

\item[\textbf{C2.}] \(\mu = (-n)\).
\[
\begin{array}{c|cc}
 & \lambda=(n,j) & \mu=(-n)\\
\hline
u & 1 & 0\\
w & 0 & 1
\end{array}
\]

Then \(u\) contains \(+n\), so it has \(+n \prec_u +j\), and \(w\) contains \(-n\), so it has \(-n \prec_w -j\). Thus \( (-j) \in \Inv(\pi(w)) \) and \( (-j) \notin \Inv(\pi(u)) \), contradicting \(\pi(w) < \pi(u)\).

\item[\textbf{C3.}] \(\mu = (n,-i)\), relation between \(i\) and \(j\) unknown.
\[
\begin{array}{c|cc}
 & \lambda=(n,j) & \mu=(n,-i)\\
\hline
u & 1 & 0\\
w & 0 & 1
\end{array}
\]
We analyze possible patterns in \(w\), depending on sign of \(n\).

\begin{itemize}
    \item If \( w \) contains \(+n \), then \( (\pm j \prec_w +n \text{ or } +n \prec_w -j )\) and \( (+n \prec_w -i )\). Thus either \( \pm j \prec_w -i\), or \(w\) contains \(-i\) and \(-j\), and in both cases \((\max\{i,j\},-\min\{i,j\}) \in \Inv(\pi(w))\).
    \item If \( w \) contains \( -n\), then \( (-n \prec_w -j) \) and \( (  \pm i \prec_w -n  \text{ or }-n \prec_w -i)\). Thus either \( \pm i \prec_w -j\), or \(w\) contains \(-i\) and \(-j\), and still in both cases \( (\max\{i,j\},-\min\{i,j\}) \in \Inv(\pi(w))\).
\end{itemize}

We now analyze possible patterns in \(u\), depending on sign of \(n\).

\begin{itemize}
    \item If \( u \) contains \(+n \), then \( (+n \prec_u +j)\) and \( ( \pm i \prec_u +n \text{ or } +n \prec_u +i )\). Thus either \( \pm i \prec_u +j\), or \(u\) contains \(+i\) and \(+j\), and in both cases \((\max\{i,j\},-\min\{i,j\}) \notin \Inv(\pi(u))\).
    \item If \( u \) contains \( -n\), then \( (\pm j \prec_u -n  \text{ or } -n \prec_u +j )\) and \( (-n \prec_u +i )\). Thus either \( \pm j \prec_u +i\), or \(u\) contains \(+i\) and \(+j\), and still in both cases \((\max\{i,j\},-\min\{i,j\}) \notin \Inv(\pi(u))\).
\end{itemize}

So \( (\max\{i,j\},-\min\{i,j\})  \in \Inv(\pi(w)) \) and \( (\max\{i,j\},-\min\{i,j\})  \notin \Inv(\pi(u)) \), contradicting \(\pi(w) < \pi(u)\). 
\end{itemize}

We have exhausted all possibilities. Hence the first differing position \(\lambda\) cannot change from \(1\) for \(u\) to \(0\) for \(w\). Therefore it changes from \(0\) for \(u\) to \(1\) for \(w\), and so \(\nu(y)>\nu(x)\).

\end{proof}
By Lemma~\ref{lem:typeB-Lemma2.2}, every edge of \(\widehat R(\pi)\) strictly increases \(\nu \); hence \(\widehat R(\pi)\) is acyclic and \(\widehat{\mathsf{Reeb}}(\pi)\) is defined.

\begin{proposition}\label{prop:typeB-auxiliary-successor}
For every vertex \(x\neq x_{\max}\), the directed edge
\(
x \longrightarrow s(x)
\)
is an edge of \(\widehat R(\pi)\).
\end{proposition}

\begin{proof}
Fix \(x=(\varepsilon,A)\neq x_{\max}\), and write \(y=s(x)\).

Comparing Definition~\ref{def:successor} with Proposition~\ref{prop:typeB-Reeb-graph}, we see:
\begin{itemize}
\item in case \textup{A1}, the edge \(x\to y\) is exactly edge type \textup{(3)} of \(R(\pi)\);
\item in case \textup{A2}, if
\[
N:=\{\,i>m:\ -i\in A\,\}=\varnothing,
\qquad m=\max([n-1]\setminus |A|),
\]
then \(y=(+,A\cup\{-m\})\), so \(x\to y\) is edge type \textup{(1)} of \(R(\pi)\);
\item in case \textup{B1}, if
\[
M:=\{\,i>k:\ i\notin |A|\,\}=\varnothing,
\qquad k=\max\{\,i\in |A|:\ +i\in A\,\},
\]
then \(y=(-,A\setminus\{+k\})\), so \(x\to y\) is edge type \textup{(2)} of \(R(\pi)\).
\end{itemize}
Hence it remains to treat nonspecial \textup{A2}, nonspecial \textup{B1}, and \textup{B2}. In each case we exhibit \(u\in x\) and \(w\in y\) such that \(\pi(w)<\pi(u)\) in \(W_{n-1}\) and \(u,w\) are incomparable in \(W_n\). This shows that \(x\to y\) is auxiliary.

\smallskip\noindent
\textbf{Case \textup{A2}, \(N\neq\varnothing\).}
Let
\[
P:=\{\,i:\ +i\in A\,\},\qquad
Q:=\{\,i<m:\ -i\in A\,\},\qquad
C:=[m-1]\setminus |A|.
\]
Write \(\mathrm{inc}\) and \(\mathrm{dec}\) for increasing and decreasing order in the signed order. Set
\[
u=\mathrm{inc}(C)\,(-m)\,(+n)\,\mathrm{inc}(P)\,\mathrm{dec}(-N)\,\mathrm{dec}(-Q),
\]
\[
w=\mathrm{inc}(C)\,\mathrm{inc}(N)\,(+n)\,(-m)\,\mathrm{inc}(P)\,\mathrm{dec}(-Q).
\]
Then \(u\in x\) and \(w\in y\).

For incomparability, for every \(i\in N\) we have \((n,-i)\in\Inv(u)\setminus\Inv(w)\), while \((n,-m)\in\Inv(w)\setminus\Inv(u)\).

Also \(\pi(w)<\pi(u)\): every inversion of \(\pi(w)\) is still an inversion of \(\pi(u)\), and \(\pi(u)\) has the extra one-letter inversions \((-i)\) for \(i\in N\). Indeed, if \(i\in N\), then inversions of \(\pi(w)\) involving \(+i\) remain inversions in \(\pi(u)\): with \(-m\) by the condition \(\pm m\prec -i\), with \(+p\) (\(p\in P\)) by \(\pm p\prec -i\), and with \(-q\) (\(q\in Q\)) by \(\pm i\prec -q\). All other relative orders are unchanged. Thus \(\Inv(\pi(w))\subsetneq \Inv(\pi(u))\).

\smallskip\noindent
\textbf{Case \textup{B1}, \(M\neq\varnothing\).}
Let
\[
C:=[k-1]\setminus |A|,
\]
and let \(\alpha\) be any word on the signed letters of \(A\setminus\{+k\}\). Set
\[
u=\mathrm{inc}(C)\,\mathrm{inc}(M)\,(-n)\,(+k)\,\alpha,
\qquad
w=\mathrm{inc}(C)\,(+k)\,(-n)\,\mathrm{inc}(M)\,\alpha.
\]
Then \(u\in x\) and \(w\in y\).

For incomparability, for every \(i\in M\) we have \((n,-i)\in\Inv(u)\setminus\Inv(w)\), while \((n,-k)\in\Inv(w)\setminus\Inv(u)\).

Also \(\pi(w)<\pi(u)\): after deleting \(\pm n\), one gets \(\pi(u)\) from \(\pi(w)\) by successively moving \(+k\) to the right across the letters \(+i\) with \(i\in M\), using the covers
\[
(+k)(+i)\longmapsto(+i)(+k)\qquad (i>k).
\]

\smallskip\noindent
\textbf{Case \textup{B2}.}
Write
\[
s_0:=\min(|A|),\qquad
A=\{-s_0\}\sqcup\{-d:\ d\in D\},
\qquad
C:=\{\,i>s_0:\ i\notin |A|\,\}.
\]
Then
\[
y=(+,A'),\qquad A'=\{+s_0\}\sqcup\{+i:\ i\in C\}.
\]
Set
\[
u=(1\,2\,\cdots\,s_0\!-\!1)\,\mathrm{inc}(C)\,(-n)\,(-s_0)\,\mathrm{inc}(-D),
\]
\[
w=\mathrm{inc}(-D)\,(1\,2\,\cdots\,s_0\!-\!1)\,(+n)\,(+s_0)\,\mathrm{inc}(C).
\]
Then \(u\in x\) and \(w\in y\).

For incomparability, \((n,-s_0)\in\Inv(u)\setminus\Inv(w)\), while \((n,s_0)\in\Inv(w)\setminus\Inv(u)\).

Also \(\pi(w)<\pi(u)\): starting from
\[
\pi(w)=\mathrm{inc}(-D)\,(1\,2\,\cdots\,s_0\!-\!1)\,(+s_0)\,\mathrm{inc}(C),
\]
first move the block \(\mathrm{inc}(-D)\) to the far right using covers \((-d)a\mapsto a(-d)\), then move \(+s_0\) to the first position using covers \(i\,s_0\mapsto s_0\,i\) for \(i<s_0\), then flip the first letter \(+s_0\mapsto -s_0\), and finally move \(-s_0\) right across \(1,2,\dots,s_0-1\) and \(\mathrm{inc}(C)\) using covers \((-s_0)a\mapsto a(-s_0)\). The result is exactly \(\pi(u)\).

Thus in every remaining case \(x\to y\) is auxiliary. Therefore \(x\to s(x)\in\widehat R(\pi)\) for every \(x\neq x_{\max}\).
\end{proof}

\begin{corollary}\label{cor:typeB-augreeb-total-order}
The augmented Reeb poset \(\widehat{\mathsf{Reeb}}(\pi)\) is the total order induced by \(\nu\):
\[
x \le_{\widehat{\mathsf{Reeb}}(\pi)} y \quad\Longleftrightarrow\quad \nu(x)\le \nu(y).
\]
\end{corollary}

\begin{proof}
If \(x\rightsquigarrow y\) is a directed path in \(\widehat R(\pi)\), then Lemma~\ref{lem:typeB-Lemma2.2} gives \(\nu(x)\le \nu(y)\).

Conversely, assume \(\nu(x)\le \nu(y)\). The successor edges form a directed path visiting all vertices exactly once, and, by Proposition~\ref{prop:succ_increases_valuation}, \(\nu\) strictly increases along this path. Thus \(x\) reaches \(y\) along successor edges. Each successor edge lies in \(\widehat R(\pi)\) by Proposition~\ref{prop:typeB-auxiliary-successor}, so \(x\le_{\widehat{\mathsf{Reeb}}(\pi)} y\).
\end{proof}

The total order \(\widehat{\mathsf{Reeb}}(\pi)\) has \(2\cdot 3^{n-1}\) elements, while the valuation \(\nu
\) takes the value \(2^{2n-1}-1\) at the maximal vertex \((-,\varnothing)\). Hence, unlike in type \(A\), the valuation \(\nu\) is not the minimal \(\mathbb Z_{\ge 0}\)-valued height function on \(\widehat{\mathsf{Reeb}}(\pi)\). Thus the cubic coordinates constructed by taking height \(\nu\) at each step  are not the minimal \(\mathbb Z_{\ge 0}\)-coordinates compatible with the tower of deletion projections.

It would be interesting to describe the minimal heights combinatorially. It already follows from the case \(n=3\) that they cannot be given by a weighted sum of the \(n\)-inversions.

Indeed, suppose that for \(n=3\) the minimal height is of the form
\[
h(x)=\sum_{\alpha\in\Inv_3(x)} w(\alpha),
\]
where \(\Inv_3(x)\) is the set of \(3\)-inversions of \(x\), and \(w\) assigns a weight to each of
\[
(3,2),\ (3,1),\ (-3),\ (3,-1),\ (3,-2).
\]

Number the elements of \(\widehat{\mathsf{Reeb}}(\pi)\) from \(0\) to \(17\) along the total order. 

\begin{center}
\begin{tabular}{c|c|l}
line & $x=(\varepsilon,A)$ & $\operatorname{Inv}_3(x)$ \\
\hline
0  & $(+,\varnothing)$            & $\varnothing$ \\
1  & $(+,\{-2\})$                 & $\{(3,-2)\}$ \\
2  & $(+,\{-1\})$                 & $\{(3,-1)\}$ \\
3  & $(+,\{-1,-2\})$              & $\{(3,-1),(3,-2)\}$ \\
4  & $(-,\{-1,-2\})$              & $\{(-3),(3,-1),(3,-2)\}$ \\
5  & $(+,\{+1\})$                 & $\{(3,1)\}$ \\
6  & $(+,\{+1,-2\})$              & $\{(3,1),(3,-2)\}$ \\
7  & $(-,\{+1,-2\})$              & $\{(-3),(3,1),(3,-2)\}$ \\
8  & $(-,\{-2\})$                 & $\{(-3),(3,1),(3,-1),(3,-2)\}$ \\
9  & $(+,\{+2\})$                 & $\{(3,2)\}$ \\
10 & $(+,\{-1,+2\})$              & $\{(3,2),(3,-1)\}$ \\
11 & $(-,\{-1,+2\})$              & $\{(3,2),(-3),(3,-1)\}$ \\
12 & $(-,\{-1\})$                 & $\{(3,2),(-3),(3,-1),(3,-2)\}$ \\
13 & $(+,\{+1,+2\})$              & $\{(3,2),(3,1)\}$ \\
14 & $(-,\{+1,+2\})$              & $\{(3,2),(3,1),(-3)\}$ \\
15 & $(-,\{+1\})$                 & $\{(3,2),(3,1),(-3),(3,-2)\}$ \\
16 & $(-,\{+2\})$                 & $\{(3,2),(3,1),(-3),(3,-1)\}$ \\
17 & $(-,\varnothing)$            & $\{(3,2),(3,1),(-3),(3,-1),(3,-2)\}$ \\
\end{tabular}
\end{center}

Then line \(1\) is \(\{(3,-2)\}\), so \(w(3,-2)=1\), and line \(2\) is \(\{(3,-1)\}\), so \(w(3,-1)=2\). Further, line \(4\) is \(\{(-3),(3,-1),(3,-2)\}\), so
\[
w(-3)+w(3,-1)+w(3,-2)=4,
\]
hence \(w(-3)=1\). Finally, line \(5\) is \(\{(3,1)\}\), so \(w(3,1)=5\).

These values agree with lines \(3\), \(6\), and \(7\). However, line \(8\) is
\[
\{(-3),(3,1),(3,-1),(3,-2)\},
\]
so the weighted sum would be
\[
w(-3)+w(3,1)+w(3,-1)+w(3,-2)=1+5+2+1=9,
\]
while the minimal height there is \(8\). 

\bibliographystyle{alpha}
\bibliography{biblio}

\end{document}